\documentclass[authoryear,review,onefignum,onetabnum]{siamart190516}

\usepackage[semicolon]{natbib}
\allowdisplaybreaks
\usepackage{adjustbox}
\usepackage{enumerate}
\usepackage{bbm}
\usepackage{float}
\usepackage[shortlabels]{enumitem}

\usepackage{bm}
\usepackage{booktabs}
\usepackage{multirow}
\usepackage{threeparttable}
\usepackage{amssymb}
\usepackage{mathrsfs}
\usepackage{actuarialsymbol}

\newcommand\Item[1][i]{%
	\ifx\relax#1\relax  \item \else \item[#1] \fi
	\abovedisplayskip=0pt\abovedisplayshortskip=0pt~\vspace*{-\baselineskip}}

\usepackage{lipsum}
\usepackage{amsfonts}
\usepackage{graphicx}
\usepackage{epstopdf}
\usepackage{algorithmic}
\ifpdf
  \DeclareGraphicsExtensions{.eps,.pdf,.png,.jpg}
\else
  \DeclareGraphicsExtensions{.eps}
\fi

\newsiamremark{remark}{Remark}
\newsiamremark{hypothesis}{Hypothesis}
\crefname{hypothesis}{Hypothesis}{Hypotheses}
\newsiamthm{claim}{Claim}

\headers{SPDE SEIRS Epidemic Models: Well-posedness and Longtime Behavior}{Y. Q. Li and L. H. Zhang}

\title{Stochastic Partial Differential Equation SEIRS Epidemic Models: Well-posedness and Longtime Behavior}

\author{Yuqi Li\thanks{School of Sciences, Beijing University of Posts and Telecommunications, Beijing, 100876, China 
  (\email{Llyq@bupt.edu.cn}).}
\and 
Lihua Zhang\thanks{Corresponding author. School of Sciences, Beijing University of Posts and Telecommunications, Beijing, 100876, China 
  (\email{zhlh@bupt.edu.cn}).}}

\usepackage{amsopn}

\def\bt{\begin{theorem}}
\def\et{\end{theorem}}
\def\bl{\begin{lemma}}
\def\el{\end{lemma}}
\def\br{\begin{remark}}
\def\er{\end{remark}}
\def\be{\begin{equation}}
\def\ee{\end{equation}}
\def\ce{\begin{equation*}}
\def\de{\end{equation*}}

\def\<{{\langle}}
\def\>{{\rangle}}

\newtheorem*{Main Theorem}[theorem]{Main Theorem}{\normalfont\bfseries}{\itshape}
{\bfseries}{\itshape} 
{\bfseries}{\itshape} 
{\bfseries}{\itshape} 

\nolinenumbers
\begin{document}

\maketitle

\begin{abstract}
The study of epidemic models plays an important role in mathematical epidemiology. There are many researches on epidemic models using ordinary differential equations, partial differential equations or stochastic differential equations. In contrast to these researches, our work analyzes the SEIRS (Susceptible-Exposed-Infected-Recovered-Susceptible) model using stochastic partial differential equations. Specifically, we consider the effects of spatial variables and space-time white noise on the epidemic model. In the means of mild solution, we construct a contraction mapping and prove the well-posedness for the stochastic partial differential equation using the fixed point argument and the factorization method. We also find sufficient conditions for permanence and extinction to determine the longtime behavior of the solution. We hope that this work can provide some new ideas for the study of increasingly complex epidemic models and related numerical simulations. 
\end{abstract}



\section{Introduction}
\label{sec1a}\vspace{-4mm}
  \indent There are many researches about epidemic models which play an important role in mathematical epidemiology \cite{WOKermack, AGMcKendrick}. In this research area, the basic model is the following SIS model \cite{LJSAllen2007}, in fact it is an ordinary differential equation. 
 \begin{equation}
 \begin{cases}
 dS(t)=[-\alpha \frac{S(t)I(t)}{S(t)+I(t)}+rI(t)]dt,\quad t>0 ,\vspace{1ex}\\
 dI(t)=[\alpha \frac{S(t)I(t)}{S(t)+I(t)}-rI(t)]dt,\quad t>0, \vspace{1ex}\\
 S(0)=S_{0} \geqslant 0, I(0)=I_{0} \geqslant 0,
 \end{cases}\notag
 \end{equation} 
 where $ S(t), I(t)$ are the densities of susceptible and infected populations, $\alpha$ is the infection rate, $r$ is the recovery rate. This basic model assume no immunity, no births and no deaths, it is actually a circular system, i.e. SISISISI $\cdots$. The more results about SIS model can refer \cite{NTDieu2019, RPeng, RPengandSLiu}.\\
 \indent Furthermore, the population also can be partitioned into susceptible, infected and recovered, and the other classical model is the following SIR model \cite{TBritton2009}, it is also an ordinary differential equation.
 \begin{equation}
 \begin{cases}
 dS(t)=[\Lambda-\mu _{S}S(t)- \alpha \frac{S(t)I(t)}{S(t)+I(t)}]dt,\quad t>0 , \vspace{1ex}\\ 
 dI(t)=[\alpha \frac{S(t)I(t)}{S(t)+I(t)}-(r+\mu _{I})I(t)]dt,\quad t>0, \vspace{1ex}\\
 dR(t)=[-\mu _{R}R(t)+rI(t)]dt, \quad t>0,\vspace{1ex}\\
 S(0)=S_{0} \geqslant 0, I(0)=I_{0} \geqslant 0, R(0)=R_{0} \geqslant 0,
 \end{cases}\notag
 \end{equation}
 in which $ R(t)$ are the densities of recovered populations, $ \Lambda$ is the recruitment rate of the population, $\mu _{S},\mu _{I},\mu _{R} $ are the death rates of susceptible, infected and recovered individuals.\\
 \indent With the developments of epidemic models, it has been widely recognized that there should be spatial dependence in the model \cite{GYin2020}, which will better reflect the spacial variations, and then it becomes a partial differential equation, such as the following epidemic reaction-diffusion model \cite{GYin2020, WWangandXQZhao, LZhang2015}.
 \begin{equation}
 \begin{cases}
 \frac{\partial }{\partial t}S(t,x)=k_{1}\triangle S(t,x) +\Lambda(x)-\mu _{1}(x)S(t,x)- \alpha(x) \frac{S(t,x)I(t,x)}{S(t,x)+I(t,x)}+r(x)I(t,x),\; \rm in \; \mathbb{R^{+}}\times \mathcal{O}  ,\vspace{1ex}\\
 \frac{\partial }{\partial t}I(t,x)=k_{2}\triangle I(t,x)-\mu _{2}(x)I(t,x) +\alpha(x) \frac{S(t,x)I(t,x)}{S(t,x)+I(t,x)}-r(x)I(t,x),\; \rm in \; \mathbb{R^{+}}\times \mathcal{O} ,\vspace{1ex} \\
 \partial_{\nu }S(t,x)= \partial_{\nu }I(t,x)=0, \; \rm in \; \mathbb{R^{+}}\times \partial \mathcal{O} ,\vspace{1ex}\\
 S(0,x)=S_{0}(x) , I(0,x)=I_{0}(x),  \; \rm in \;\mathcal{O}  ,
 \end{cases}\notag
 \end{equation}	
 where $\triangle$ is the Laplacian with respect to the spatial variable, $\mathcal{O}$ is a bounded domain with $ C^{2}$ boundary of $\mathbb{R}^{d} (d\geqslant1) $, $	\partial_{\nu }S, 	\partial_{\nu }I$ denote the directional derivative with the $ \nu$ being outer normal direction on $ \partial \mathcal{O}$, and $ k_{1}$ and $ k_{2}$ are positive constants representing the diffusion rates of the susceptible and infected population densities, respectively. In addition, $\Lambda(x), \mu _{1}(x), \mu _{2}(x), \alpha(x), r(x)$ are non-negative functions.\\
 \indent On the one hand, the above models are all noise free, however, random noise perturbations in the environment often inevitably appear. Recently, one popular approach is adding stochastic noise perturbations to the above deterministic models, there are more and more researches for stochastic epidemic models \cite{DHNguyen2016, NHDuandNNNhu1, NHDuandNNNhu2, NTHieu2015}. On the other hand, in addition to$ S(t), I(t), R(t)$, the exposed populations $ E(t)$ also could be considered. Naturally, how can we consider SEIRS (Susceptible-Exposed-Infected-Recovered-Susceptible) epidemic models that contain both spatial variables $ x$ and space-time white noise$ W(t,x)$? \\
 \indent In this paper, our work analyzes the SEIRS model using stochastic partial differential equations. Specifically, we consider the effects of spatial variables and space-time white noise on the epidemic model. In the means of mild solution, we construct a contraction mapping and prove the well-posedness for the stochastic partial differential equation using the fixed point argument and the factorization method in Section 3.1. In Section 3.2, we find sufficient conditions for permanence and extinction to determine the longtime behavior of the solution. In Section 2, we give the notations and assumptions, and Section 4 concludes the paper. We hope that this work can provide some new ideas for the study of increasingly complex epidemic models and related numerical simulations.
 
 \section{Preliminary}
 
 \subsection{Notations and Assumptions}
 \noindent  $\bullet$ $ \mathcal{L}$: a bounded domain with $ C^{2}$ boundary of $ \mathbb{R}^{d}, d \geqslant 1$. Assume $ \left| \mathcal{L}\right|=1$, where $ \left| \mathcal{L}\right|$ is the volume of $ \mathcal{L}$ in $\mathbb{R}^{d}$, and the initial values are non-random for simplicity. \\
 $\bullet$ $ (\Omega ,\mathcal{F}, \left\{\mathcal{F}_{t}\right\}_{t\geqslant 0},\mathbb{P})$: a complete filtered probability space. \\
 $\bullet$ $ L^{p}\bigl(\Omega; C\bigl([0,T];C(\mathcal{\bar{L}};\mathbb{R}^{4})\bigr)\bigr)$: the space of all predictable $ C(\mathcal{\bar{L}};\mathbb{R}^{4})$-valued processes in $C\bigl([0,T];C(\mathcal{\bar{L}};\mathbb{R}^{4})\bigr)$, $\mathbb{P}$-a.s. with the norm $L_{t,p}$ as follows
 \begin{equation}
 \left| u\right|_{L_{t,p}}^{p}:=\mathbb{E} \displaystyle \sup_{ s\in [0,t]}\left| u(s)\right|_{C(\mathcal{\bar{L}};\mathbb{R}^{4})}^{p}, \notag
 \end{equation}\vspace{1ex}
 \quad where 
 \begin{equation}
 \left| u\right|_{C(\mathcal{\bar{L}};\mathbb{R}^{4})}=\Bigl( \sum_{i=1}^{4}\displaystyle \sup_{ x\in \mathcal{\bar{L}}}\left| u_{i}(x)\right|^{2} \Bigr)^{\frac{1}{2}}. \notag
 \end{equation}
 $\bullet$ Define $  L^{2}(\mathcal{L}; \mathbb{R})$-value Wiener processes
 \begin{equation}
 W_{i}(t)=\sum_{k=1}^{\infty }\sqrt{a_{k,i}}B_{k,i}(t)e_{k}, i=1,2,3,4 , \notag
 \end{equation}
 where $ \left\{ a_{k,i}\right\}_{k=1}^{\infty }$ are sequences of non-negative real numbers satisfying $a_{i}:=\sum_{k=1}^{\infty }a_{k,i}<\infty ,i=1,2,3,4$, $\left\{ e_{k}\right\}_{k=1}^{\infty }$ are orthonormal basis.\\
 $\bullet$ Define $ H$ is Banach space $ C(\mathcal{\bar{L}};\mathbb{R}) $, \\
 \indent  $ C_{0}:= \displaystyle \sup_{ k \in \mathbb{N}}\left| e_{k}\right|_{L^{\infty }( \mathcal{L};\mathbb{R})}=\sup_{ k \in \mathbb{N}} \underset{x \in \mathcal{L}}{\rm ess sup}e_{k}(x)<\infty$, \quad $ \left|u \right|_{H}:=\underset{x \in \mathcal{\bar{L}}}{\rm sup}\left| u(x)\right|$. \\
 $\bullet$ Let $A_{i}$ be Neumann realizations of $k_{i}\triangle$ in $ H$, $i=1,2,3,4$, $ A=(A_{1},A_{2}, A_{3},A_{4})$.\\

 \subsection{SPDE SEIRS Epidemic Model}
 We propose a SEIRS Epidemic Model using a system of Stochastic Partial Differential Equations given by
 \begin{equation}\label{1}
 \begin{cases}
 dS(t,x)=[k_{1}\triangle S(t,x) +\Lambda(x)-\mu _{1}(x)S(t,x)-\alpha(x) \frac{S(t,x)I(t,x)}{S(t,x)+I(t,x)+E(t,x)+R(t,x)}\\ \quad \quad \quad \quad+\beta(x)R(t,x)]dt+S(t,x)dW_{1}(t,x),\quad \rm in \; \mathbb{R^{+}}\times\mathcal{L}  ,\\
 dE(t,x)=[k_{2}\triangle E(t,x) -\mu _{2}(x)E(t,x)+ \alpha(x) \frac{S(t,x)I(t,x)}{S(t,x)+I(t,x)+E(t,x)+R(t,x)}\\ \quad \quad \quad \quad-\sigma(x)E(t,x)]dt+E(t,x)dW_{2}(t,x),\quad \rm in \; \mathbb{R^{+}}\times\mathcal{L} ,\\		
 dI(t,x)=[k_{3}\triangle I(t,x) -\mu _{3}(x)I(t,x)+\sigma(x)E(t,x)-\gamma(x)I(t,x)]dt\\ \quad \quad \quad \quad+I(t,x)dW_{3}(t,x),\quad \rm in \; \mathbb{R^{+}}\times\mathcal{L}  ,\\		
 dR(t,x)=[k_{4}\triangle R(t,x) -\mu _{4}(x)R(t,x)+\gamma(x)I(t,x)-\beta(x)R(t,x)]dt\\ \quad \quad \quad \quad+R(t,x)dW_{4}(t,x),\quad \rm in \; \mathbb{R^{+}}\times \mathcal{L}  ,\\	   
 \partial_{\nu }S(t,x)= \partial_{\nu }E(t,x)=\partial_{\nu }I(t,x)=\partial_{\nu }R(t,x)=0, \quad \rm in \; \mathbb{R^{+}}\times \partial \mathcal{L} \\
 S(0,x)=S_{0}(x) , E(0,x)=E_{0}(x),  I(0,x)=I_{0}(x),  R(0,x)=R_{0}(x) \quad \rm in \; \mathcal{L}  ,
 \end{cases} 
 \end{equation}
 where $\triangle$ is the Laplacian with respect to the spatial variable, $	\partial_{\nu }S, \partial_{\nu }E, \partial_{\nu }I,\partial_{\nu }R$ denote the directional derivative with the $ \nu$ being outer normal direction on $ \partial \mathcal{L}$, $ k_{1}, k_{2}, k_{3}$ and $ k_{4}$ are positive constants representing the diffusion rates of the susceptible, exposed, infected and recovered population densities, $ \mu _{1}(x), \mu _{2}(x), \mu _{3}(x) $ and $ \mu _{4}(x)$ represents the death rates of above population densities, respectively. $\Lambda(x) $ is the recruitment rate of the population, $ \alpha(x)$ is the infection rate, $ \gamma(x)$ is the recovery rate, $ \beta(x)$ is the proportion of recovered individuals who lost immunity,
 $ \sigma(x)$ is the proportion of exposed individuals who become infected. In addition, $\Lambda(x), \mu _{1}(x), \mu _{2}(x), \mu _{3}(x), \mu _{4}(x), \alpha(x), \gamma(x), \beta(x), \sigma(x)$ are non-negative functions. $W_{1}(t,x), W_{2}(t,x)$, $W_{3}(t,x)$ and $W_{4}(t,x)$ are $  L^{2}(\mathcal{L}; \mathbb{R})$-value Wiener processes, they are also called Brownian sheet, and in some significants, their derivatives are space-time white noise.\\
 \indent  Define $ Au:=(A_{1}u_{1}, A_{2}u_{2}, A_{3}u_{3}, A_{4}u_{4})$, in which $ u=(u_{1}, u_{2}, u_{3}, u_{4}) \in L^{2}(\mathcal{L}; \mathbb{R}^{4})$, then $ A$ is infinitesimal generators of analytic semi-groups $e^{tA} $, and $e^{tA}u=(e^{tA_{1}}u_{1}, e^{tA_{2}}u_{2}, e^{tA_{3}}u_{3}, e^{tA_{4}}u_{4} ) $. We can rewrite equation (\ref{1}) as the stochastic differential equations in an infinite dimension space
 \begin{equation}
 \begin{cases} 
 dS(t)=[A_{1}S(t)+\Lambda-\mu _{1}S(t)- \frac{\alpha S(t)I(t)}{S(t)+I(t)+E(t)+R(t)}+\beta R(t)]dt+S(t)dW_{1}(t)  ,\vspace{1ex}\\
 dE(t)=[A_{2}E(t)-\mu _{2}E(t)+\frac{\alpha S(t)I(t)}{S(t)+I(t)+E(t)+R(t)}-\sigma E(t)]dt+E(t)dW_{2}(t)  ,\vspace{1ex}\\
 dI(t)=[A_{3}I(t)-\mu _{3}I(t)+\sigma E(t)-\gamma I(t))]dt+I(t)W_{3}(t),\vspace{1ex}\\
 dR(t)=[A_{4}E(t)-\mu _{4}R(t)+\gamma I(t)-\beta R(t))]dt+R(t)W_{4}(t), \vspace{1ex}\\
 S(0)=S_{0}, E(0)=E_{0},  I(0)=I_{0},  R(0)=R_{0}.\label{2}
 \end{cases} 
 \end{equation}
 \begin{remark}
 	This doesn't mean the equation (\ref{2}) is independent of $x$, it is just written in the form of an infinite dimension dynamics system.
 \end{remark}
 \begin{definition}
 	$ (S(t), E(t), I(t), R(t)) $ is a \textbf{mild solution} to (\ref{2}), if 
 	\begin{equation}
 	\begin{cases} 
 	S(t)=e^{tA_{1}}S_{0}+\int_{0}^{t}e^{(t-s)A_{1}}(\Lambda-\mu _{1}S(s)- \frac{\alpha S(s)I(s)}{S(s)+I(s)+E(s)+R(s)}+\beta R(s))ds+W_{S}(t)  ,\vspace{0.5ex}\\
 	E(t)=e^{tA_{2}}E_{0}+\int_{0}^{t}e^{(t-s)A_{2}}(-\mu _{2}E(s)+ \frac{\alpha S(s)I(s)}{S(s)+I(s)+E(s)+R(s)}-\sigma E(s))ds+W_{E}(t),\vspace{0.5ex}\\
 	I(t)=e^{tA_{3}}I_{0}+\int_{0}^{t}e^{(t-s)A_{3}}(-\mu _{3}I(s)+\sigma E(s)-\gamma I(s))ds+W_{I}(t),\vspace{0.5ex}\\
 	R(t)=e^{tA_{4}}R_{0}+\int_{0}^{t}e^{(t-s)A_{4}}(-\mu _{4}R(s)+\gamma I(s)-\beta R(s))ds+W_{R}(t), \label{m}
 	\end{cases} 
 	\end{equation}
 	where
 	\begin{equation}
 	W_{S}(t)=\int_{0}^{t}e^{(t-s)A_{1}}S(s) dW_{1}(t), \quad W_{E}(t)=\int_{0}^{t}e^{(t-s)A_{2}}E(s) dW_{2}(t), \notag
 	\end{equation}	
 	\begin{equation}
 	W_{I}(t)=\int_{0}^{t}e^{(t-s)A_{3}}I(s) dW_{3}(t), \quad W_{R}(t)=\int_{0}^{t}e^{(t-s)A_{4}}R(s) dW_{4}(t). \notag
 	\end{equation}
 	\indent Define a \textbf{positive mild solution} of (\ref{2}) as a mild solution $ S(t), E(t), I(t), R(t)$ such that $ S(t), E(t)$, $I(t), R(t) \geqslant 0$, almost everywhere $ x \in \mathcal{L}$, for all $ t \geqslant 0$. In this paper, we only consider positive mild solution. Furthermore, to make the term $\frac{si}{s+i+e+r}$ well defined, assume that $ \frac{si}{s+i+e+r}=0$ when $ s=0$ or $i=0$.
 \end{definition}
 
 \indent To make the analysis more clear an convenient, we rewrite the mild solution (\ref{m}) in the following vector form:
 \begin{equation}
 V(t)=e^{tA}V_{0}+\int_{0}^{t}e^{(t-s)A}G(V(s))ds+\int_{0}^{t}e^{(t-s)A}(V(s))dW(s), \label{*}
 \end{equation}
 in which 
 \begin{equation}
 V=(S,E,I,R), \quad V_{0}=(S_{0},E_{0},I_{0},R_{0}), \notag
 \end{equation}
 \begin{equation}
 \begin{split}
 &G(V)=(G_{1}(V), G_{2}(V), G_{3}(V), G_{4}(V))\\ :=&\Bigl(\Lambda-\mu _{1}S- \frac{\alpha SI}{S+I+E+R}+\beta R, -\mu _{2}E+ \frac{\alpha SI}{S+I+E+R}-\sigma E, -\mu _{3}I+\sigma E-\gamma I, -\mu _{4}R+\gamma I-\beta R \Bigr), \notag
 \end{split}
 \end{equation}
 \begin{equation}
 \begin{split}
 &e^{(t-s)A}(V(s))dW(s)\\:=&\Bigl(e^{(t-s)A_{1}}(S(s))dW_{1}(s), e^{(t-s)A_{2}}(E(s))dW_{2}(s), e^{(t-s)A_{3}}(I(s))dW_{3}(s), e^{(t-s)A_{4}}(R(s))dW_{4}(s)\Bigr). \notag
 \end{split}
 \end{equation}

 \begin{definition}
 	The infected category is said to be \textbf{permanent}, if there exists a positive number $ R$, is independent of initial values, such that
 	\begin{equation}
 	\displaystyle \liminf_{ t\to \infty }\frac{1}{t}\int_{0}^{t}\Bigl(\mathbb{E}\int _{\mathcal{O}}\bigl((I(s,x)+E(s,x))^{2} \wedge 1\bigr)dx\Bigr)^{\frac{1}{2}}ds\geqslant R, \notag	
 	\end{equation} 
 	and that is said to be \textbf{extinct}, if 
 	\begin{equation}
 	\displaystyle \limsup_{ t\to \infty }\frac{1}{t}\int_{0}^{t}\mathbb{E}\int _{\mathcal{O}}(I(s,x)+E(s,x))dxds=0. \notag	
 	\end{equation} 
 	
 \end{definition}
 
 \section{Main Results}
 \indent In this section, we give our main results: in the means of mild solution, we construct a contraction mapping and prove the well-posedness for the stochastic partial differential equation using the fixed point argument and the factorization method. We also find sufficient conditions for permanence and extinction to determine the longtime behavior of the solution.. 
 
 \subsection{Well-posedness}
 \begin{theorem}
 	There exists a unique positive mild solution $ (S(t), E(t), I(t), R(t)) $ of (\ref{2}), which is in
 	$ L^{p}\bigl(\Omega; C\bigl([0,T];C(\mathcal{\bar{L}};\mathbb{R}^{4})\bigr)\bigr)$, for any $ T>0, p \geqslant 1$ and any initial value $ 0 \leqslant S_{0}, E_{0}, I_{0}, R_{0} \in C(\mathcal{\bar{L}};\mathbb{R}) $. Moreover, this solution depends on the initial value continuously.	
 \end{theorem}        
 
 \noindent \textbf{Outline of Proof of Theorem 3.1}:\\	
 $\blacktriangleright$ Step1. Construct a contraction mapping \\
 $\blacktriangleright$ Step2. Prove the positivity of the mild solution\\
 $\blacktriangleright$ Step3. Prove the continuous dependence of the solution on initial values.\\
 
 \noindent \textit{\textbf{Proof of Theorem 3.1}}:\\
 \noindent \textbf{Step1}. \textbf{Construct $ \tilde{\gamma}$ and prove it is a contraction mapping}.\\
 \indent Define
 \begin{equation}
 \begin{split}
 g(x,s,e,i,r)=\Bigl(\Lambda(x)&-\mu_{1}(x)s-\frac{\alpha (x)si}{s+e+i+r}+\beta(x)r,\; -\mu_{2}(x)e+\frac{\alpha (x)si}{s+e+i+r}-\sigma(x)e, \;\\&-\mu_{3}(x)i+\sigma(x)e-\gamma(x)i,\;-\mu_{4}(x)r+\gamma(x)i-\beta(x)r  \Bigr), \notag
 \end{split}
 \end{equation}
 \begin{equation}
 g^{*}(x,s,e,i,r)=f(x,s\vee 0,e \vee 0,i \vee 0,r \vee 0), \quad v=(s,e,i,r).\notag
 \end{equation}
 \indent According to the previous assumption, obviously we can obtain $g^{*}(x,\cdot ,\cdot ,\cdot ,\cdot ): \mathbb{R}^{4}\to \mathbb{R}^{4} $ is Lipschitz continuous, uniformly in $ x \in \mathcal{O}$. Therefore, $G^{*}(v)(x)= \bigl( G_{1}^{*}(v)(x), G_{2}^{*}(v)(x), G_{3}^{*}(v)(x), G_{4}^{*}(v)(x) \bigr):=g^{*}(x,v(x)), x \in \mathcal{L} $ is Lipschitz continuous, both in $  L^{2}(\mathcal{L}; \mathbb{R}^{4})$ and $C(\mathcal{\bar{L}};\mathbb{R}^{4}) $.\\
 \indent Now, we consider the following problem
 \begin{equation}
 dV^{*}(t)=[AV^{*}(t)+G^{*}(V^{*}(t))]dt+(V^{*}(t) \vee 0)dW(t), \; V^{*}(0)=V_{0}=(S_{0},E_{0},I_{0},R_{0}), \label{3}
 \end{equation}
 where
 \begin{equation}
 V^{*}(t)=\bigl(S^{*}(t), E^{*}(t), I^{*}(t), R^{*}(t)\bigr), \notag
 \end{equation}
 \begin{equation}
 (V^{*}(t) \vee 0)(x)=\bigl(S^{*}(t,x) \vee 0, E^{*}(t,x) \vee 0, I^{*}(t,x) \vee 0, R^{*}(t,x) \vee 0\bigr). \notag
 \end{equation}
 For any $u(t,x)=\Bigl(u_{1}(t,x), u_{2}(t,x), u_{3}(t,x), u_{4}(t,x)\Bigr) \in  L^{p}\Bigl(\Omega; C\bigl([0,T];C(\mathcal{\bar{L}};\mathbb{R}^{4})\bigr)\Bigr)$, we consider the mapping 
 \begin{equation}
 \tilde{\gamma}(u)(t):=e^{tA}V_{0}+\int_{0}^{t}e^{(t-s)A}G^{*}(u(s))ds+\tilde{\varphi } (u)(t), \notag
 \end{equation}
 where
 \begin{equation}
 \begin{split}
 \tilde{\varphi } (u)(t):&=\int_{0}^{t}e^{(t-s)A}(u(s) \vee 0)dW(s)\\
 &=\Bigl(  \int_{0}^{t}e^{(t-s)A_{1}}(u_{1}(s) \vee 0)dW_{1}(s) ,\int_{0}^{t}e^{(t-s)A_{2}}(u_{2}(s) \vee 0)dW_{2}(s) , \\&\int_{0}^{t}e^{(t-s)A_{3}}(u_{3}(s) \vee 0)dW_{3}(s) , \int_{0}^{t}e^{(t-s)A_{4}}(u_{4}(s) \vee 0)dW_{4}(s)    \Bigr). \notag
 \end{split}
 \end{equation}	
 \indent  \textbf{If} we can prove $ \tilde{\gamma}$ is a contraction mapping in $ L^{p}\bigl(\Omega; C\bigl([0,T_{0}];C(\mathcal{\bar{L}};\mathbb{R}^{4})\bigr)\bigr)$, for some $ T_{0}>0$, and any $ p \geqslant p_{0}$ for some $ p_{0}$, then by a fixed point argument we obtain that equation (\ref{3}) admits a unique mild solution. Thus, by repeating the above argument in each finite time interval $ [kT_{0},(k+1)T_{0}]$, for any $ T>0, p \geqslant p_{0}$ the equation (\ref{3}) admits a unique solution $ V^{*}(t)=\bigl(S^{*}(t), E^{*}(t), I^{*}(t), R^{*}(t)\bigr)$ in $ L^{p}\bigl(\Omega; C\bigl([0,T];C(\mathcal{\bar{L}};\mathbb{R}^{4})\bigr)\bigr)$. The key point is to estimate $ \tilde{\varphi } (u)(t)$. We need the following lemma to complete the proof of contraction mapping.
 \begin{lemma}
 	There exists $ p_{0}$ such that for any $ p \geqslant p_{0}$, the mapping $ \tilde{\varphi }$ maps $ L^{p}\bigl(\Omega; C\bigl([0,t];C(\mathcal{\bar{L}};\mathbb{R}^{4})\bigr)\bigr)$ into itself, and for any $ u=(u_{1}, u_{2}, u_{3}, u_{4}), \; v=(v_{1}, v_{2}, v_{3}, v_{4})$ $ \in L^{p}\bigl(\Omega; C\bigl([0,t];C(\mathcal{\bar{L}};\mathbb{R}^{4})\bigr)\bigr) $,
 	\begin{equation}
 	\left|\tilde{\varphi }(u)-\tilde{\varphi } (v)\right|_{L_{t,p}}\leqslant c_{p}(t)\left| u-v\right|_{L_{t,p}}, \notag
 	\end{equation}
 	where $ c_{p}(t)$ is some constant satisfying $ c_{p}(t) \downarrow 0$ as $ t\downarrow 0$.
 \end{lemma}
 \noindent $ \bigstar$ \textbf{The main ingredients for the proof of Lemma 3.1}\\
 
 \indent Due to $ \tilde{\varphi } (u)(t)$ is not a stochastic integral, the martingale method is not valid, so we need to use the factorization method which was first gave by Da Prato and Zabczyk \cite{GDaPratoandJZabczyk, GDaPratoandLTubaro} to prove our results.\\
 Based on the equality
 \begin{equation}
 \int_{\sigma }^{t}(t-s)^{\alpha -1}(s-\sigma )^{-\alpha }ds=\frac{\pi }{\rm sin\pi \alpha }, \notag	
 \end{equation} 
 by the stochastic Fubini theorem, we have
 \begin{equation}
 \tilde{\varphi }(u)(t)-\tilde{\varphi } (v)(t)=\frac{\rm sin\pi \alpha  }{\pi } \int_{0}^{t}(t-s)^{\alpha -1}e^{(t-s)A}M_{\alpha }(u,v)(s)ds,\notag
 \end{equation} 
 where
 \begin{equation}
 M_{\alpha }(u,v)(s)=\int_{0}^{s} (s-r)^{-\alpha }e^{(s-r)A}\bigl(u(r)\vee 0-v(r)\vee 0\bigr)dW(r)	.\notag
 \end{equation}
 Then, using BDG inequality and the properties of semi-group \cite{WArendt, SCerrai}, easy to prove
 \begin{equation}
 \int_{0}^{t}\left|M_{\alpha }(u,v)(s) \right|^{p}_{L^{p}(\mathcal{L};\mathbb{R}^{4})}ds<\infty  \quad \rm a.s. . \notag
 \end{equation}
 Therefore, we can apply Holder inequality and Sobolev embedding theorem to obtain
 \begin{equation}
 \textcolor{black}{	\left|\tilde{\varphi }(u)(t)-\tilde{\varphi } (v)(t)\right|_{L_{t,p}} \leqslant c_{ p}(t)\left|u-v\right|_{L_{t,p}}}. \notag
 \end{equation} 
 Lemma 3.1 is proved.\\
 \indent Using the properties of analytic semi-group \cite{EBDavies, EMOuhabaz} and the Lipschitz continuous of $ G^{*}$, we obtain
 \begin{equation}
 \begin{split}
 \int_{0}^{t}\left|e^{(s-t)A}(G^{*}(u(s))-G^{*}(v(s))) \right|^{p}_{L(\bar{\mathcal{L}};\mathbb{R}^{4})}ds&\leqslant c\int_{0}^{t}\left|u(s)-v(s) \right|^{p}_{C(\bar{\mathcal{L}};\mathbb{R}^{4})} \\&\leqslant c\int_{0}^{t}\displaystyle \sup_{ r\in [0,s]}\left|u(s)-v(s) \right|^{p}_{C(\bar{\mathcal{L}};\mathbb{R}^{4})}\\& \leqslant ct\cdot \sup_{ s\in [0,t]}\left|u(s)-v(s) \right|^{p}_{C(\bar{\mathcal{L}};\mathbb{R}^{4})}. \label{f}
 \end{split}
 \end{equation}
 Combining (\ref{f}) and Lemma 3.1, 
 \begin{equation}
 \left|  \tilde{\gamma}(u)- \tilde{\gamma}(v)\right|_{L_{t,p}}\leqslant c_{p}(t)\left| u-v\right|_{L_{t,p}}, \notag
 \end{equation} 
 in which $ c_{p}(t)\downarrow 0(t\downarrow 0) $.\\
 \indent We have proved $ \tilde{\gamma}$ is a contraction mapping in $ L^{p}\bigl(\Omega; C\bigl([0,T_{0}];C(\mathcal{\bar{L}};\mathbb{R}^{4})\bigr)\bigr)$, for some $ T_{0}>0$. Therefore, Step1 has been completed.
 
 \noindent $\blacktriangleright$ \textbf{Step2}.  \textbf{Prove the positivity of $ S^{*}(t), E^{*}(t), I^{*}(t), R^{*}(t)$}.\\
 $ \bullet$ For $ i=1,2,3,4$, let $ \lambda_{i} \in \rho (A_{i})$ be the resolvent set of $ A_{i}$ and $ R_{i}(\lambda_{i}):= \lambda_{i}R_{i}(\lambda_{i},A_{i})$, with $ R_{i}(\lambda_{i},A_{i})$ being the resolvent of $ A_{i}$. For each small $ \varepsilon >0$, there exists a unique strong solution \cite{K. Liu} $S_{\lambda ,\varepsilon }(t,x), E_{\lambda ,\varepsilon }(t,x),$ $ I_{\lambda ,\varepsilon }(t,x)$, $R_{\lambda ,\varepsilon }(t,x) $ of the equation
 \begin{equation}
 \begin{cases} 
 dS_{\lambda ,\varepsilon }(t)=[A_{1}S_{\lambda ,\varepsilon }(t)+R_{1}(\lambda _{1})F_{1}\left ( \varepsilon \Phi (\varepsilon ^{-1}S_{\lambda ,\varepsilon }(t)),  \varepsilon \Phi (\varepsilon ^{-1}E_{\lambda ,\varepsilon }(t)),  \varepsilon \Phi (\varepsilon ^{-1}I_{\lambda ,\varepsilon }(t)), \varepsilon \Phi (\varepsilon ^{-1}R_{\lambda ,\varepsilon }(t))\right )]dt\\ \quad \qquad \qquad \qquad \qquad+R_{1}(\lambda _{1})\varepsilon \Phi (\varepsilon ^{-1}S_{\lambda ,\varepsilon }(t))dW_{1}(t)  ,\vspace{1ex}\\
 
 dE_{\lambda ,\varepsilon }(t)=[A_{2}E_{\lambda ,\varepsilon }(t)+R_{2}(\lambda _{2})F_{2}\left ( \varepsilon \Phi (\varepsilon ^{-1}S_{\lambda ,\varepsilon }(t)),  \varepsilon \Phi (\varepsilon ^{-1}E_{\lambda ,\varepsilon }(t)),  \varepsilon \Phi (\varepsilon ^{-1}I_{\lambda ,\varepsilon }(t)), \varepsilon \Phi (\varepsilon ^{-1}R_{\lambda ,\varepsilon }(t))\right )]dt\\ \quad \qquad \qquad \qquad \qquad+R_{2}(\lambda _{2})\varepsilon \Phi (\varepsilon ^{-1}E_{\lambda ,\varepsilon }(t))dW_{2}(t)  ,\vspace{1ex}\\
 
 dI_{\lambda ,\varepsilon }(t)=[A_{3}I_{\lambda ,\varepsilon }(t)+R_{3}(\lambda _{3})F_{3}\left ( \varepsilon \Phi (\varepsilon ^{-1}S_{\lambda ,\varepsilon }(t)),  \varepsilon \Phi (\varepsilon ^{-1}E_{\lambda ,\varepsilon }(t)),  \varepsilon \Phi (\varepsilon ^{-1}I_{\lambda ,\varepsilon }(t)), \varepsilon \Phi (\varepsilon ^{-1}R_{\lambda ,\varepsilon }(t))\right )]dt\\ \quad \qquad \qquad \qquad \qquad+R_{3}(\lambda _{3})\varepsilon \Phi (\varepsilon ^{-1}I_{\lambda ,\varepsilon }(t))dW_{3}(t)  ,\vspace{1ex}\\
 
 dR_{\lambda ,\varepsilon }(t)=[A_{4}R_{\lambda ,\varepsilon }(t)+R_{4}(\lambda _{4})F_{4}\left ( \varepsilon \Phi (\varepsilon ^{-1}S_{\lambda ,\varepsilon }(t)),  \varepsilon \Phi (\varepsilon ^{-1}E_{\lambda ,\varepsilon }(t)),  \varepsilon \Phi (\varepsilon ^{-1}I_{\lambda ,\varepsilon }(t)), \varepsilon \Phi (\varepsilon ^{-1}R_{\lambda ,\varepsilon }(t))\right )]dt\\ \quad \qquad \qquad \qquad \qquad+R_{4}(\lambda _{4})\varepsilon \Phi (\varepsilon ^{-1}R_{\lambda ,\varepsilon }(t))dW_{4}(t)  ,\vspace{1ex}\\
 
 S_{\lambda ,\varepsilon }(0)=R_{1}(\lambda _{1})S_{0}, E_{\lambda ,\varepsilon }(0)=R_{2}(\lambda _{2})E_{0},  I_{\lambda ,\varepsilon }(0)=R_{3}(\lambda _{3})I_{0},  R_{\lambda ,\varepsilon }(0)=R_{4}(\lambda _{4})R_{0}. \notag
 \end{cases} 
 \end{equation}
 
 where 
 \begin{equation}
 \Phi(\xi)=\begin{cases} 
 \displaystyle 0, &\xi \leqslant 0,\\
 \displaystyle 3\xi ^{5}-8\xi ^{4}+6\xi ^{3},&0<\xi <1,\\
 \displaystyle \xi,&\xi \geqslant 1,
 \end{cases}  \notag
 \end{equation}		
 the construction of $\Phi$ is unique and satisfies $ \Phi\in C^{2}(\mathbb{R}),\varepsilon \Phi (\varepsilon ^{-1}\xi )\to \xi \vee 0\;(\varepsilon \to 0) $.\\
 \indent For some sequence $ \left\{\lambda (k) \right\}_{k=1}^{\infty }$, we obtain that
 \begin{equation}
 \left(S_{\lambda(k) ,\varepsilon }(t,x), E_{\lambda(k) ,\varepsilon }(t,x), I_{\lambda(k) ,\varepsilon }(t,x), R_{\lambda(k) ,\varepsilon }(t,x)\right) \rightarrow \bigl(S^{*}(t), E^{*}(t), I^{*}(t), R^{*}(t)\bigr) \quad (k \rightarrow \infty, \varepsilon \rightarrow 0). \notag
 \end{equation}
 \indent Let 
 \begin{equation}
 g(\xi )=\begin{cases}
 2\xi ^{2}+\frac{11}{3},\quad \xi \leqslant -1 \\
 -\xi ^{4}-\frac{8}{3}\xi ^{3},\quad -1<\xi <0\\
 0, \quad \xi \geqslant 0 \notag
 \end{cases}
 \end{equation}
 which satisfies $ g^{\prime }(\xi) \leqslant 0$ and $ g^{\prime \prime}(\xi) \geqslant 0$, the construction of $ g(\xi)$ is not unique. \\
 \indent For $ \forall \xi $, we have $ g^{\prime }(\xi)	\Phi(\xi)=g^{\prime \prime}(\xi)\Phi(\xi)=0$, then by Ito s formula in \cite{RFCurtain}, we get	
 \begin{equation}
 \int_{\mathcal{O}}g(S_{\lambda ,\varepsilon }(t,x))dx\leqslant 0 ,\int_{\mathcal{O}}g(E_{\lambda ,\varepsilon }(t,x))dx\leqslant 0, \int_{\mathcal{O}}g(I_{\lambda ,\varepsilon }(t,x))dx\leqslant 0,\int_{\mathcal{O}}g(R_{\lambda ,\varepsilon }(t,x))dx\leqslant 0. \notag
 \end{equation}
 Therefor, we have the positivity of $ S^{*}(t), E^{*}(t), I^{*}(t), R^{*}(t)$.
 
 \noindent $\blacktriangleright$ \textbf{Step3}. \textbf{Prove the continuous dependence of the solution on initial values}.\\
 \indent We use subscripts to indicate the dependence of the solution on initial value. Let $ V_{v_{0}}(t), V_{v_{0}^{'}}(t)$ be the positive mild solution of (\ref{*}) with the initial condition $ V(0)=v_{0}$ and $ V(0)=v_{0}^{'}$, then some estimates are given
 \begin{equation}
 \begin{split}
 &\mathbb{E}\displaystyle \sup_{ s \in [0,t]}\left|\int_{0}^{s}e^{(s-r)A}(V_{v_{0}}(r)-V_{v_{0}^{'}}(r))dW(r) \right|_{C(\mathcal{\bar{L}};\mathbb{R}^{4})}^{p}\\
 &\leqslant c_{p}(t)\int_{0}^{t} \mathbb{E}\displaystyle \sup_{ r \in [0,s]}\left| V_{v_{0}}(r)-V_{v_{0}^{'}}(r) \right|_{C(\mathcal{\bar{L}};\mathbb{R}^{4})}^{p}ds\\
 &\leqslant c_{p}(t) \int_{0}^{t}\left| V_{v_{0}}-V_{v_{0}^{'}} \right|_{L_{s,p}}^{p}ds, \notag
 \end{split}
 \end{equation}
 \begin{equation}
 \left| V_{v_{0}}-V_{v_{0}^{'}} \right|_{L_{t,p}}^{p} \leqslant c_{p}\left| V_{v_{0}}-V_{v_{0}^{'}} \right|_{C(\mathcal{\bar{L}};\mathbb{R}^{4})}^{p}	+ c_{p}(t) \int_{0}^{t}\left| V_{v_{0}}-V_{v_{0}^{'}} \right|_{L_{s,p}}^{p}ds. \notag
 \end{equation}
 By Gronwall's inequality, 
 \begin{equation}
 \left| V_{v_{0}}-V_{v_{0}^{'}} \right|_{L_{T,p}}^{p} \leqslant c_{p}(T)\left| v_{0}-v_{0}^{'} \right|_{C(\mathcal{\bar{L}};\mathbb{R}^{4})}^{p}. \notag
 \end{equation}
 $  \hfill\square $
 \subsection{Longtime Behavior}
 \indent Define $\Lambda _{*}=\displaystyle \inf_{ x\in \mathcal{\bar{L}}}\Lambda (x)$, $\hat{R}=\int _{\mathcal{L}}\frac{\alpha (x)}{2}-\bigl(\mu _{2}(x)+\mu _{3}(x)+\gamma (x)\bigr)dx-\frac{\tilde{a}}{2}$, where $ \tilde{a}= \rm max\left\{ a_{2},a_{3}\right\} $. 
 \begin{theorem}
 	If $ \Lambda _{*}>0$, $ \hat{R}>0$, and for all $ t \leqslant 0$,		
 	\begin{equation}
 	\begin{split}
 	\alpha(x)- \gamma(x) \geqslant 0, \quad & \frac{S(t,x)+R(t,x)}{I(t,x)+E(t,x)}\leqslant \frac{S(t,x)+R(t,x)+I(t,x)+E(t,x)}{2},\\&\frac{S(t,x)}{S(t,x)+R(t,x)+I(t,x)+E(t,x)}>\frac{\gamma(x) }{\alpha(x) }, \notag
 	\end{split}
 	\end{equation}
 	almost everywhere $ x \in \mathcal{L}$.	Then the \textbf{infected category} is \textbf{permanent} for any initial values $ 0 \leqslant S_{0}, E_{0}, I_{0}, R_{0} \in C(\mathcal{\bar{L}};\mathbb{R}) $ satisfying
 	\begin{equation}
 	\int _{\mathcal{L}}-ln(I_{0}(x)+E_{0}(x))dx < \infty, \notag
 	\end{equation} 
 	we also obtain
 	\begin{equation}
 	\displaystyle \liminf_{ t\to \infty }\frac{1}{t}\int_{0}^{t}\Bigl(\mathbb{E}\int _{\mathcal{L}}\bigl((I(s,x)+E(s,x))^{2}\wedge 1\bigr)dx\Bigr)^{\frac{1}{2}}ds\geqslant R>0, \notag	
 	\end{equation}
 	in which $ R$ is independent of initial values.
 \end{theorem}
 \noindent \textit{\textbf{Proof of Theorem 3.2}}:\\
 \noindent \textbf{Step1}. \textbf{Using the strong solution to approximate the mild solution}.\\ 	
 \indent For each fixed $ n \in \mathbb{N}$, let $ (S_{n}(t,x), E_{n}(t,x), I_{n}(t,x), R_{n}(t,x))$ be the strong solution \cite{E.B. Davies,E.M. Ouhabaz} of the following equations
 \begin{equation}
 \begin{cases}
 dS_{n}(t,x)=[A_{1}S_{n}(t,x) +\Lambda(x)-\mu _{1}(x)S_{n}(t,x)-\alpha(x) \frac{S_{n}(t,x)I_{n}(t,x)}{S_{n}(t,x)+I_{n}(t,x)+E_{n}(t,x)+R_{n}(t,x)}\\ \quad \quad \quad \quad+\beta(x)R_{n}(t,x)]dt+\sum_{k=1}^{n}\sqrt{a_{k,1}}e_{k}(x)S_{n}(t,x)dB_{k,1}(t)  ,\\
 dE_{n}(t,x)=[A_{2}E_{n}(t,x) -\mu _{2}(x)E_{n}(t,x)+ \alpha(x) \frac{S_{n}(t,x)I_{n}(t,x)}{S_{n}(t,x)+I_{n}(t,x)+E_{n}(t,x)+R_{n}(t,x)}\\ \quad \quad \quad \quad-\sigma(x)E_{n}(t,x)]dt+\sum_{k=1}^{n}\sqrt{a_{k,2}}e_{k}(x)E_{n}(t,x)dB_{k,2}(t),\\		
 dI_{n}(t,x)=[A_{3}I_{n}(t,x) -\mu _{3}(x)I_{n}(t,x)+\sigma(x)E_{n}(t,x)-\gamma(x)I_{n}(t,x)]dt\\ \quad \quad \quad \quad+\sum_{k=1}^{n}\sqrt{a_{k,3}}e_{k}(x)I_{n}(t,x)dB_{k,3}(t),\\		
 dR_{n}(t,x)=[A_{4}R_{n}(t,x) -\mu _{4}(x)R_{n}(t,x)+\gamma(x)I_{n}(t,x)-\beta(x)R_{n}(t,x)]dt\\ \quad \quad \quad \quad+\sum_{k=1}^{n}\sqrt{a_{k,4}}e_{k}(x)R_{n}(t,x)dB_{k,4}(t),\\	   
 S_{n}(0,x)=S_{0}(x) , E_{n}(0,x)=E_{0}(x),  I_{n}(0,x)=I_{0}(x),  R_{n}(0,x)=R_{0}(x)  .\label{s}
 \end{cases} 
 \end{equation}	
 \indent By the conclusions about stochastic partial equations in \cite{G. Da Prato and J. Zabczyk,N. N. Nguyen and G. Yin}, the strong solution of (\ref{s}) is uniqueness, and for all fixed $t$, 
 \begin{equation}
 \displaystyle \lim_{n \to \infty }\mathbb{E}\left|S(t)-S_{n}(t) \right|^{2}_{H}=0, \displaystyle \lim_{n \to \infty }\mathbb{E}\left|E(t)-E_{n}(t) \right|^{2}_{H}=0,\notag
 \end{equation} 
 \begin{equation}
 \displaystyle \lim_{n \to \infty }\mathbb{E}\left|I(t)-I_{n}(t) \right|^{2}_{H}=0, \displaystyle \lim_{n \to \infty }\mathbb{E}\left|R(t)-R_{n}(t) \right|^{2}_{H}=0. \notag
 \end{equation}
 \noindent \textbf{Step2}. \textbf{Prove the following lemmas}.	
 \begin{lemma}
 	Let $ \mu_{*}=\displaystyle \inf_{ x\in \mathcal{\bar{L}}} min {\mu_{1}(x), \mu_{2}(x), \mu_{3}(x), \mu_{4}(x)}$, if $\mu_{*}>0 $, then 
 	\begin{equation}
 	\mathbb{E}\int _{\mathcal{L}}(S(t,x)+I(t,x)+E(t,x)+R(t,x))dx \leqslant \int _{\mathcal{L}}(S_{0}(x)+I_{0}(x)+E_{0}(x)+R_{0}(x))dx+\frac{\left|\Lambda  \right|_{H}}{\mu _{*}}e^{\mu _{*}t}. \notag
 	\end{equation}
 \end{lemma}
 The proof of Lemma 3.2 is similar to \cite{G. Yin2020}, we omit it.
 \begin{lemma}
 	If $\int_{\mathcal{L}}\frac{1}{(S_{0}(x)+R_{0}(x))^{p}}dx<\infty, \quad  \forall p>0  $, there exists $ \tilde{M}_{p}>0$, such that
 	\begin{equation}
 	\mathbb{E} \int_{\mathcal{L}}\frac{1}{(S_{n}(t,x)+R_{n}(t,x))^{p}}dx \leqslant e^{-t}\int_{\mathcal{L}}\frac{1}{(S_{0}(x)+R_{0}(x))^{p}}dx+ \tilde{M}_{p}, \notag
 	\end{equation}
 	in which $ \tilde{M}_{p}$ is independent of $n$ and initial conditions. 
 \end{lemma}
 \begin{proof}
 	\rm For $e^{t}\int_{\mathcal{L}}\frac{1}{(S_{n}(t,x)+R_{n}(t,x)+\epsilon)^{p}}dx $, using Ito's formula in \cite{R.F. Curtain}, we have
 	\begin{equation}
 	\begin{split}
 	&d\Bigl(e^{t}\int_{\mathcal{L}}\frac{1}{(S_{n}(t,x)+R_{n}(t,x)+\epsilon)^{p}}dx \Bigr)=
 	e^{t}\int_{\mathcal{L}}\frac{1}{(S_{n}(t,x)+R_{n}(t,x)+\epsilon)^{p}}dx \\
 	+ &e^{t}\int_{\mathcal{L}}\frac{-p}{(S_{n}(t,x)+R_{n}(t,x)+\epsilon)^{p+1}}\Bigl[k_{1}\triangle S_{n}(t,x) +\Lambda(x)-\mu _{1}(x)S_{n}(t,x)\\-&\alpha(x) \frac{S_{n}(t,x)I_{n}(t,x)}{S_{n}(t,x)+I_{n}(t,x)+E_{n}(t,x)+R_{n}(t,x)} +\beta(x)R_{n}(t,x)+k_{4}\triangle R_{n}(t,x) \\-&\mu _{4}(x)R_{n}(t,x)+\gamma(x)I_{n}(t,x)-\beta(x)R_{n}(t,x)\Bigr]dxdt\\
 	+&e^{t}\int_{\mathcal{L}}\frac{-p}{(S_{n}(t,x)+R_{n}(t,x)+\epsilon)^{p+1}}\sum_{k=1}^{n}\sqrt{a_{k,1}}e_{k}(x)S_{n}(t,x)dxdB_{k,1}(t)\\
 	+&e^{t}\int_{\mathcal{L}}\frac{-p}{(S_{n}(t,x)+R_{n}(t,x)+\epsilon)^{p+1}}\sum_{k=1}^{n}\sqrt{a_{k,4}}e_{k}(x)R_{n}(t,x)dxdB_{k,4}(t)\\
 	+&\frac{e^{t}}{2}\int_{\mathcal{L}}\frac{p(p+1)\sum_{k=1}^{n}a_{k,1}}{(S_{n}(t,x)+R_{n}(t,x)+\epsilon)^{p+2}}e^{2}_{k}(x)S^{2}_{n}(t,x)dxdt\\
 	+&\frac{e^{t}}{2}\int_{\mathcal{L}}\frac{p(p+1)\sum_{k=1}^{n}a_{k,4}}{(S_{n}(t,x)+R_{n}(t,x)+\epsilon)^{p+2}}e^{2}_{k}(x)R^{2}_{n}(t,x)dxdt.  \notag
 	\end{split}
 	\end{equation}
 	Let 
 	\begin{equation}
 	m= \rm max\left\{\left|\mu_{1} \right|_{H}, \left|\mu_{4} \right|_{H}, \left|\alpha-\gamma \right|_{H} \right\}, \quad M_{p}=\frac{1}{p}+2m+\frac{p+1}{2}aC^{2}_{0}. \notag
 	\end{equation}
 	Due to the assumptions, after deflating we can get
 	\begin{equation}
 	\begin{split}
 	&e^{t}\int_{\mathcal{L}}\frac{1}{(S_{n}(t,x)+R_{n}(t,x)+\epsilon)^{p}}dx \\\leqslant
 	&\int_{\mathcal{L}}\frac{1}{(S_{0}(x)+R_{0}(x)+\epsilon)^{p}}dx \\
 	+&\sum_{k=1}^{n}\int_{0}^{t}e^{s}\left [ \sqrt{a_{k,1}} \int_{\mathcal{L}}\frac{-pe_{k}S_{n}(t,x)}{(S_{n}(t,x)+R_{n}(t,x)+\epsilon)^{p+1}}dx\right ]dB_{k,1}(s)\\
 	+&\sum_{k=1}^{n}\int_{0}^{t}e^{s}\left [ \sqrt{a_{k,4}} \int_{\mathcal{L}}\frac{-pe_{k}R_{n}(t,x)}{(S_{n}(t,x)+R_{n}(t,x)+\epsilon)^{p+1}}dx\right ]dB_{k,4}(s)\\
 	+&\int_{0}^{t}\frac{p\cdot M^{p+1}_{p}\cdot 2^{p}}{\Lambda ^{p}_{*}}e^{s}ds. \label{ee}
 	\end{split}
 	\end{equation}
 	Therefore, (\ref{ee}) implies that 
 	\begin{equation}
 	\begin{split}
 	&\mathbb{E}\int_{\mathcal{L}}\frac{1}{(S_{n}(t,x)+R_{n}(t,x)+\epsilon)^{p}}dx\\\leqslant
 	&e^{-t}\int_{\mathcal{L}}\frac{1}{(S_{0}(x)+R_{0}(x)+\epsilon)^{p}}dx\\
 	+&e^{-t}\int_{0}^{t}\frac{p\cdot M^{p+1}_{p}\cdot 2^{p}}{\Lambda ^{p}_{*}}e^{s}ds.
 	\end{split}
 	\end{equation}
 	Define $ \tilde{M}_{p}=e^{-t}\int_{0}^{t}\frac{p\cdot M^{p+1}_{p}\cdot 2^{p}}{\Lambda ^{p}_{*}}e^{s}ds$, and letting $ \epsilon \rightarrow 0$, we complete the proof.
 	$  \hfill\square $
 \end{proof}
 
 \begin{lemma}
 	\begin{equation}
 	\mathbb{E} \int_{\mathcal{L}}\frac{1}{(S_{n}(4,x)+R_{n}(4,x))^{2}}dx \leqslant k_{0}, \quad n \in \mathbb{N}, \notag
 	\end{equation}
 	in which $k_{0}$ is independent of $ n$, it depends only initial condition.
 \end{lemma}
 \begin{proof}
 	\rm Similar to the proof of Lemma 3.3, we have the following facts
 	\begin{equation}
 	\mathbb{E}\int_{\mathcal{L}}(S_{n}(t,x)+R_{n}(t,x))dx \leqslant \int_{\mathcal{L}}(S_{0}(x)+R_{0}(x))dx+\left|\Lambda  \right|_{H}\cdot t. \notag
 	\end{equation}
 	Then, there exists $ k_{1}>0$ ($k_{1}$ is independent of $ n$) s.t 
 	\begin{equation}
 	\mathbb{E}\int_{\mathcal{L}}(S_{n}(t,x)+R_{n}(t,x))^{q}dx \leqslant k_{1}, \quad t \in [0,1], q \in [0,1] . \label{e1}	 
 	\end{equation}
 	\indent Let $ N_{1}= \frac{1}{2}(m+\left| \alpha \right|_{H}+\frac{1}{4}aC_{0}^{2})$, we obtain
 	\begin{equation}
 	\begin{split}
 	&\mathbb{E}\int_{\mathcal{L}}(S_{n}(t,x)+R_{n}(t,x)+\epsilon)^{\frac{1}{2}}dx \\ \geqslant
 	& \frac{1}{2}\int_{0}^{1}\mathbb{E}\int_{\mathcal{L}}\frac{\Lambda(x) }{(S_{n}(1,x)+R_{n}(1,x)+\epsilon)^{\frac{1}{2}}}dxds\\
 	-&N_{1}\int_{0}^{1}\mathbb{E}\int_{\mathcal{L}}(S_{n}(t,x)+R_{n}(t,x))^{\frac{1}{2}}dxds.
 	\end{split} \label{e2}
 	\end{equation}
 	Due to (\ref{e1}) and (\ref{e2}), 
 	\begin{equation}
 	\int_{0}^{1}\mathbb{E}\int_{\mathcal{L}}\frac{\Lambda(x) }{(S_{n}(s,x)+R_{n}(s,x))^{\frac{1}{2}}}dxds \leqslant 2(1+N_{1})k_{1},  \notag
 	\end{equation}
 	equivalently, there exists $ t_{1}=t_{1}(n) \in [0,1]$ s.t.
 	\begin{equation}
 	\mathbb{E}\int_{\mathcal{L}}\frac{\Lambda(x) }{(S_{n}(t_{1},x)+R_{n}(t_{1},x))^{\frac{1}{2}}}dx \leqslant 2(1+N_{1})k_{1} . \notag
 	\end{equation}
 	\indent Applying Lemma 3.3 and the Markov property of $(S_{n}(t,x), E_{n}(t,x), I_{n}(t,x), R_{n}(t,x)) $, we get
 	\begin{equation}
 	\mathbb{E}\int_{\mathcal{L}}\frac{\Lambda(x) }{(S_{n}(t,x)+R_{n}(t,x)+\epsilon)^{\frac{1}{2}}}dx \leqslant k_{2}, \quad \forall t \in [1,2], \notag
 	\end{equation}
 	in which $ k_{2}$ is independent of $ n$.
 	Using Ito s formula in \cite{R.F. Curtain}, we have
 	\begin{equation}
 	\begin{split}
 	&\mathbb{E}\int_{\mathcal{L}}(S_{n}(2,x)+R_{n}(2,x)+\epsilon)^{-\frac{1}{2}}dx \\ \leqslant
 	&\mathbb{E}\int_{\mathcal{L}}(S_{n}(1,x)+R_{n}(1,x)+\epsilon)^{-\frac{1}{2}}dx\\
 	-&\frac{1}{2}\int_{1}^{2}\mathbb{E}\int_{\mathcal{L}}\frac{\Lambda(x) }{(S_{n}(s,x)+R_{n}(s,x)+\epsilon)^{\frac{3}{2}}}dxds\\
 	+&\frac{1}{2}(m+\left| \alpha \right|_{H}+\frac{3}{4}aC_{0}^{2})\int_{1}^{2}\mathbb{E}\int_{\mathcal{L}}(S_{n}(s,x)+R_{n}(s,x))^{-\frac{1}{2}}dxds.
 	\end{split} \notag
 	\end{equation}
 	\indent Repeat the above method, we have the following facts.\\
 	(1) There exists $ t_{2}=t_{2}(n) \in [1,2]$ s.t.
 	\begin{equation}
 	\mathbb{E}\int_{\mathcal{L}}\frac{\Lambda(x) }{(S_{n}(t_{2},x)+R_{n}(t_{3},x))^{\frac{3}{2}}}dx \leqslant k_{3}, \notag
 	\end{equation}
 	then, 
 	\begin{equation}
 	\mathbb{E}\int_{\mathcal{L}}\frac{\Lambda(x) }{(S_{n}(t,x)+R_{n}(t,x)+\epsilon)^{\frac{3}{2}}}dx \leqslant k_{4}, \quad \forall t \in [2,3]. \notag
 	\end{equation}
 	(2) There exists $ t_{3}=t_{3}(n) \in [3,4]$ s.t.
 	\begin{equation}
 	\mathbb{E}\int_{\mathcal{L}}\frac{\Lambda(x) }{(S_{n}(t_{2},x)+R_{n}(t_{3},x))^{\frac{5}{2}}}dx \leqslant k_{5} . \notag
 	\end{equation}
 	Thus, there exists $ t_{4}=t_{4}(n) \in [0,4]$ s.t.
 	\begin{equation}
 	\mathbb{E}\int_{\mathcal{L}}\frac{\Lambda(x)}{(S_{n}(t_{4},x)+R_{n}(t_{4},x))^{2}}dx \leqslant k_{6} . \notag
 	\end{equation}
 	Finally, applying the Lemma 3.3, Lemma 3.4 is proved.
 	$  \hfill\square $
 \end{proof}
 \noindent \textbf{Step3}. \textbf{Complete the proof of Theorem 3.2}.
 By the Lemma 3.3 and Lemma 3.4, we get
 \begin{equation}
 \mathbb{E} \int_{\mathcal{L}}\frac{1}{(S_{n}(t,x)+R_{n}(t,x))^{2}}dx \leqslant e^{-t}k_{0}+\tilde{M}, \quad \forall n \in \mathbb{N}, t \geqslant 4, \label{d1}
 \end{equation}
 in which $ k_{0}$ and $ \tilde{M}$ are both independent of $n$, but $ k_{0}$ depend on initial value.\\
 \indent Define $\tilde{m}= \rm max\left\{\left|\mu_{2} \right|_{H}, \left|\mu_{3} \right|_{H}, \left|\gamma \right|_{H} \right\}$. After deflating, we can obtain the following estimates.\\
 \begin{equation}
 \mathbb{E}\int_{\mathcal{L}}(I_{n}(t,x)+E_{n}(t,x))dx \geqslant\mathbb{E}\int_{\mathcal{L}} ln(I_{n}(t,x)+E_{n}(t,x))dx \geqslant \int_{\mathcal{L}}(I_{0}(x)+E_{0}(x))dx-(\tilde{m}+\tilde{a})> -\infty,\quad  \forall t >0, \notag
 \end{equation}
 i.e. \\
 \begin{equation}
 P\left\{(I_{n}(t,x)+E_{n}(t,x))>0 \quad a.e. \; \rm in \; \mathcal{L} \right\}=1,  \quad \forall t >0, \forall n \in \mathbb{N}, \label{d2}
 \end{equation}
 \begin{equation}
 \begin{split}
 &\int_{0}^{t}\mathbb{E}\int_{\mathcal{L}}\Bigl(\frac{\alpha(x)(I_{n}(t,x)+E_{n}(t,x))}{(S_{n}(t,x)+R_{n}(t,x))(I_{n}(t,x)+E_{n}(t,x))}dxds \\
 &+\int_{0}^{t}\mathbb{E}\int_{\mathcal{L}}\Bigl(\frac{\alpha(x)(S_{n}(t,x)+R_{n}(t,x))(I_{n}(t,x)+E_{n}(t,x))}{(S_{n}(t,x)+R_{n}(t,x)+I_{n}(t,x)+E_{n}(t,x))(I_{n}(t,x)+E_{n}(t,x)+1)}\Bigr)dxds\\ \geqslant
 &\int_{\mathcal{L}}ln\frac{I_{0}(x)+E_{0}(x)+\epsilon}{1+I_{0}(x)+E_{0}(x)}dx+\hat{R}t-\left|\alpha  \right|_{H}\int_{0}^{t}\int_{\mathcal{L}}\frac{\epsilon}{I_{n}(s,x)+E_{n}(s,x)+\epsilon}dxds, \quad \forall t \ \geqslant 0, \forall n \in \mathbb{N}, 0<\epsilon<1, \label{d3}
 \end{split}
 \end{equation}
 (\ref{d3}) implies 
 \begin{equation}
 \begin{split}
 &\int_{0}^{t}\mathbb{E}\int_{\mathcal{L}}\Bigl(\frac{\alpha(x)(I_{n}(t,x)+E_{n}(t,x))}{(S_{n}(t,x)+R_{n}(t,x))(I_{n}(t,x)+E_{n}(t,x))}dxds \\
 &+\int_{0}^{t}\mathbb{E}\int_{\mathcal{L}}\Bigl(\frac{\alpha(x)(S_{n}(t,x)+R_{n}(t,x))(I_{n}(t,x)+E_{n}(t,x))}{(S_{n}(t,x)+R_{n}(t,x)+I_{n}(t,x)+E_{n}(t,x))(I_{n}(t,x)+E_{n}(t,x)+1)}\Bigr)dxds\\ \geqslant
 &\int_{\mathcal{L}}ln\frac{I_{0}(x)+E_{0}(x)}{1+I_{0}(x)+E_{0}(x)}dx+\hat{R}t, \quad \forall t \ \geqslant 0, \forall n \in \mathbb{N},  \label{d4}
 \end{split}
 \end{equation}
 due to the facts
 \begin{equation}
 \frac{S_{n}(t,x)+R_{n}(t,x)}{S_{n}(t,x)+R_{n}(t,x)+I_{n}(t,x)+E_{n}(t,x)} \leqslant 1, \notag
 \end{equation}
 \begin{equation}
 \frac{1+E_{n}(t,x)+I_{n}(t,x)}{S_{n}(t,x)+R_{n}(t,x)+I_{n}(t,x)+E_{n}(t,x)} \leqslant \frac{1}{S_{n}(t,x)+R_{n}(t,x)}+1, \notag
 \end{equation}
 we have
 \begin{equation}
 \begin{split}
 &\left|\alpha  \right|_{H} \Bigl(\mathbb{E}\int_{\mathcal{L}}\frac{(I_{n}(t,x)+E_{n}(t,x))^{2}}{(1+I_{n}(t,x)+E_{n}(t,x))^{2}}dx\Bigr)^{\frac{1}{2}} \\\geqslant& \mathbb{E}\int_{\mathcal{L}}\frac{\alpha(x)(I_{n}(t,x)+E_{n}(t,x))}{I_{n}(t,x)+E_{n}(t,x)+1}dx \\\geqslant& \mathbb{E}\int_{\mathcal{L}}\frac{\alpha(x)(S_{n}(t,x)+R_{n}(t,x))(I_{n}(t,x)+E_{n}(t,x))}{(S_{n}(t,x)+R_{n}(t,x)+I_{n}(t,x)+E_{n}(t,x))(I_{n}(t,x)+E_{n}(t,x)+1)}dx ,\label{d5}
 \end{split}
 \end{equation}
 \begin{equation}
 \begin{split}
 &\left|\alpha  \right|_{H} \Bigl(\mathbb{E}\int_{\mathcal{L}}\frac{(I_{n}(t,x)+E_{n}(t,x))^{2}}{(1+I_{n}(t,x)+E_{n}(t,x))^{2}}dx\Bigr)^{\frac{1}{2}} \cdot \Bigl(\mathbb{E}\int_{\mathcal{L}}\bigl(\frac{1}{S_{n}(t,x)+R_{n}(t,x)}+1\bigr)^{2}dx\Bigr)^{\frac{1}{2}}\\ \geqslant& \mathbb{E}\int_{\mathcal{L}}\frac{\alpha(x)(I_{n}(t,x)+E_{n}(t,x))}{S_{n}(t,x)+R_{n}(t,x)+I_{n}(t,x)+E_{n}(t,x)}dx
 \end{split}
 \end{equation}
 Combining the above estimates, we can obtain
 \begin{equation}
 \begin{split}
 &\int_{4}^{t}\left|\alpha  \right|_{H} \Bigl(\mathbb{E}\int_{\mathcal{L}}\frac{(I_{n}(s,x)+E_{n}(s,x))^{2}}{(1+I_{n}(s,x)+E_{n}(s,x))^{2}}dx\Bigr)^{\frac{1}{2}} \left ( 2\sqrt{e^{-s}k_{0}} +2\tilde{M}^{\frac{1}{2}}+3\right )ds\\\geqslant& \int_{\mathcal{L}}ln\frac{I_{0}(x)+E_{0}(x)}{1+I_{0}(x)+E_{0}(x)}dx+\hat{R}t-8\left|\alpha  \right|_{H} ,\notag
 \end{split}
 \end{equation}
 letting $ n \rightarrow \infty$, then 
 \begin{equation}
 \begin{split}
 &\int_{4}^{t}\left|\alpha  \right|_{H} \Bigl(\mathbb{E}\int_{\mathcal{L}}\frac{(I(s,x)+E(s,x))^{2}}{(1+I(s,x)+E(s,x))^{2}}dx\Bigr)^{\frac{1}{2}} \left ( 2\sqrt{e^{-s}k_{0}} +2\tilde{M}^{\frac{1}{2}}+3\right )ds\\\geqslant& \int_{\mathcal{L}}ln\frac{I_{0}(x)+E_{0}(x)}{1+I_{0}(x)+E_{0}(x)}dx+\hat{R}t-8\left|\alpha  \right|_{H} ,
 \end{split}
 \end{equation}
 equivalently, 
 \begin{equation}
 \begin{split}
 &\underset{t\rightarrow \infty }{\rm liminf}\frac{1}{t}\int_{0}^{t}\left (\mathbb{E}\int_{\mathcal{L}} ((I(s,x)+E(s,x))^{2}\wedge 1)dx \right )^{\frac{1}{2}}ds \\ \geqslant &
 \underset{t\rightarrow \infty }{\rm liminf}\frac{1}{t}\int_{0}^{t}\left (\mathbb{E}\int_{\mathcal{L}} \frac{(I_{n}(s,x)+E_{n}(s,x))^{2}}{(1+I_{n}(s,x)+E_{n}(s,x))^{2}}dx \right )^{\frac{1}{2}}ds \\ \geqslant & \frac{\hat{R}}{\left|\alpha  \right|_{H}(2\tilde{M}^{\frac{1}{2}}+3)} >0. \notag
 \end{split}
 \end{equation}
 Therefore, we complete the proof of Theorem 3.2. $  \hfill\square $
 
 \begin{theorem}
 	If 
 	\begin{equation}
 	(\mu_{3}+\gamma-\alpha) _{*}=\displaystyle \inf_{ x\in \mathcal{\bar{O}}}(\mu_{3}(x)+\gamma(x)-\alpha(x))>0, \quad (\mu_{2}) _{*}=\displaystyle \inf_{ x\in \mathcal{\bar{O}}}\mu_{2}(x)>0, \notag
 	\end{equation}
 	and $ 0 \leqslant S_{0}, E_{0}, I_{0}, R_{0} \in C(\mathcal{\bar{L}};\mathbb{R}) $,
 	then the \textbf{infected category} will be \textbf{extinct} with exponential rate.
 \end{theorem}
 \begin{proof}
 	\rm Define the following operator $J: L^{2}(\mathcal{L};\mathbb{R}) \rightarrow \mathbb{R}$,
 	\begin{equation}
 	Ju:= \int_{\mathcal{L}}u(x)dx, \quad \forall u \in  L^{2}(\mathcal{L};\mathbb{R}). \notag
 	\end{equation}
 	Thus, we can obtain
 	\begin{equation}
 	\begin{split}
 	&\int_{\mathcal{L}}(I(t,x)+E(t,x))dx=\int_{\mathcal{L}}(I_{0}(x)+E_{0}(x))dx\\+&\int_{0}^{t}\int_{\mathcal{L}}\Bigl(-\mu _{2}(x)E(s,x)+ \alpha(x) \frac{S(s,x)I(s,x)}{S(s,x)+I(s,x)+E(s,x)+R(s,x)} -\mu _{3}(x)I(s,x)-\gamma(x)I(s,x)\Bigr)dxds \\+&\int_{0}^{t}J(e^{(t-s)A_{2}}E(s))dW_{2}(s) +\int_{0}^{t}J(e^{(t-s)A_{3}}I(s))dW_{3}(s) .\label{i1}
 	\end{split}
 	\end{equation}
 	(\ref{i1}) implies 
 	\begin{equation}
 	\begin{split}
 	&\mathbb{E}\int_{\mathcal{L}}(I(t,x)+E(t,x))dx=\int_{\mathcal{L}}(I_{0}(x)+E_{0}(x))dx\\+&\int_{0}^{t}\mathbb{E}\int_{\mathcal{L}}\Bigl(-\mu _{2}(x)E(s,x)+ \alpha(x) \frac{S(s,x)I(s,x)}{S(s,x)+I(s,x)+E(s,x)+R(s,x)} -\mu _{3}(x)I(s,x)-\gamma(x)I(s,x)\Bigr)dxds .
 	\end{split}
 	\end{equation}
 	Let $ m= \rm min\left\{(\mu _{3}+\gamma -\alpha )_{*}, (\mu _{2})_{*} \right\}$, then we have
 	\begin{equation}
 	\begin{split}
 	&\mathbb{E}\int_{\mathcal{L}}(I(t,x)+E(t,x))dx-\mathbb{E}\int_{\mathcal{L}}(I(s,x)+E(s,x))dx\\=&
 	\int_{s}^{t}\mathbb{E}\int_{\mathcal{L}}\Bigl(-\mu _{2}(x)E(r,x)+ \alpha(x) \frac{S(r,x)I(s,x)}{S(r,x)+I(r,x)+E(r,x)+R(r,x)} -\mu _{3}(x)I(r,x)-\gamma(x)I(r,x)\Bigr)dxdr \\ \leqslant&
 	\int_{s}^{t}\mathbb{E}\int_{\mathcal{L}}\Bigl(-\mu _{2}(x)E(r,x) -\bigl(\mu _{3}(x)+\gamma(x)- \alpha(x)\bigr)I(r,x)\Bigr)dxdr \\ \leqslant &
 	-m \int_{s}^{t}\mathbb{E}\int_{\mathcal{L}}(I(r,x)+E(r,x))dxdr. \label{i2}
 	\end{split}
 	\end{equation}
 	(\ref{i2}) implies
 	\begin{equation}
 	\frac{d}{dt^{+}}\mathbb{E}\int_{\mathcal{L}}(I(t,x)+E(t,x))dx \leqslant -m\mathbb{E}\int_{\mathcal{L}}(I(t,x)+E(t,x))dx, \notag 
 	\end{equation}
 	that means for some $ C>0$, 
 	\begin{equation}
 	\mathbb{E}\int_{\mathcal{L}}(I(t,x)+E(t,x))dx \leqslant Ce^{-mt} , \notag
 	\end{equation}
 	then the infected category will be extinct with exponential rate.
 	$  \hfill\square $
 \end{proof}
 
 \section{Summary}
 \indent In this paper, we considered the effects of spatial variables and space-time white noise on the SEIRS epidemic model. In the means of mild solution, we proved the well-posedness for the stochastic partial differential equation. Furthermore, we found sufficient conditions for permanence and extinction to determine the longtime behavior of the solution.\\
 \indent In our future work, we would like to consider the stochastic regularity of the solution in this SEIRS epidemic model, and the more optimal conditions of longtime behavior are worthy to consider. It is also possible to give the numerical analysis about the epidemic models which contain the spatial variables.
 
\bibliography{references}
\end{document}